\documentclass[10pt,a4paper]{amsart}

\usepackage{enumerate}
\usepackage[english]{babel}
\usepackage{amssymb,amsthm}

\usepackage{amsmath}

\usepackage[pdftex,plainpages=false,pdfpagelabels]{hyperref}

\numberwithin{equation}{section}
\theoremstyle{plain}
\newtheorem{thm}{Theorem}

\newtheorem{lem}[thm]{Lemma}

\theoremstyle{definition}

\theoremstyle{remark}
\newtheorem{rem}[thm]{Remark}

\newcommand{\sL}{\mathcal{L}}
\newcommand{\sD}{\mathcal{D}}

\newcommand{\Id}{I}

\newcommand{\CC}{\mathbb C}
\newcommand{\NN}{\mathbb N}
\newcommand{\RR}{\mathbb R}
\newcommand{\e}{{\rm e}}

\newcommand{\ds}{\,ds}

\renewcommand{\l}{\left}
\renewcommand{\r}{\right}

\begin{document}
\parindent0cm

\title[Differential equations for stochastic network processes]
{Differential equation approximations of stochastic network processes: an operator semigroup approach}
\author[A.~B\'atkai]{Andr\'as B\'atkai}
\author[I.~Z.~Kiss]{Istvan Z.~Kiss}
\author[E.~Sikolya]{Eszter Sikolya}
\author[P.~L.~Simon]{P\'eter L.~Simon}

\address{E\"otv\"os Lor\'and University Budapest, Hungary}
\address{School of Mathematical and Physical Sciences, Department of Mathematics, University of
Sussex, Falmer, Brighton BN1 9RF, UK}
\email{batka@cs.elte.hu, seszter@cs.elte.hu, simonp@cs.elte.hu}
\email{i.z.kiss@sussex.ac.uk}
\date{\today}
\subjclass{}%
\keywords{dynamic network, birth-and-death process, one-parameter operator semigroup}%

\begin{abstract}

The rigorous linking of exact stochastic models to mean-field
approximations is studied. Starting from the differential equation
point of view the stochastic model is identified by its Kolmogorov
equations, which is a system of linear ODEs that depends on the state
space size ($N$) and can be written as $\dot u_N=A_N u_N$. Our
results rely on the convergence of the transition matrices $A_N$
to an operator $A$. This convergence also implies that the
solutions $u_N$ converge to the solution $u$ of $\dot u=Au$. The
limiting ODE can be easily used to derive simpler mean-field-type
models such that the moments of the stochastic process will
converge uniformly to the solution of appropriately chosen
mean-field equations. A bi-product of this method is the proof
that the rate of convergence is $\mathcal{O}(1/N)$. In addition,
it turns out that the proof holds for cases that are slightly more
general than the usual density dependent one. Moreover, for Markov
chains where the transition rates satisfy some sign conditions, a
new approach for proving convergence to the mean-field limit is
proposed. The starting point in this case is the derivation of a
countable system of ordinary differential equations for all the
moments. This is followed by the proof of a perturbation theorem
for this infinite system, which in turn leads to an estimate for
the difference between the moments and the corresponding
quantities derived from the solution of the mean-field ODE.
\end{abstract}
\maketitle

\section{Introduction}

\label{sec:int}

\noindent General birth-and-death models are at the basis of many real-world applications ranging
from queuing theory to disease transmission models, see Grimmett and Stirzaker \cite{GS:01}. In
particular, the analysis of such models involves the consideration of Kolmogorov equations that
simply describe the evolution of the probability of a certain process being in a given state at a
given time. One of the major drawbacks of this approach is the large number of equations. This is
very limiting from an analysis viewpoint, and in addition, it also precludes the construction of a
numerical solution of the full or complete set of equations. Using techniques such as lumping in
Simon et al. \cite{STK:11} this can be circumvented and an equivalent exact system with a
tractable number of equations can be derived. However, often this technique only works in the
presence of significant system symmetries such as in the case of a simple disease transmission
model on a fully connected graph where all nodes are topologically identical. This requirement
rarely holds and highlights the importance of approaches that deal with the complexity of the
large number of equations. Progress in this direction has been made as illustrated by the
important contributions of Kurtz \cite{EK:05}, Bobrowski \cite{Bob:05} and Darling and Norris
\cite{DANo:08}.

Here we take a dynamical system type approach, where the
Kolmogorov equations are simply considered as a system of linear
ODEs with a transition rate matrix having specific properties such
as special tri-diagonal structure and/or well defined functional
form for the transmission rates. For example, consider a Markov
chain with finite state space $\{ 0, 1, \ldots ,N\}$ and denote by
$p_k(t)$ the probability that the system is in state $k$ at time
$t$ (with a given initial state that is not specified at the
moment). Assuming that starting from state $k$ the system can move
to either state $k-1$ or to state $k+1$, the Kolmogorov equations
of the Markov chain take the form
\begin{equation}
    \dot{p}_k=\beta_{k-1}p_{k-1}-\alpha_{k}p_k+\delta_{k+1}p_{k+1},\quad k=0,\ldots, N,
    \tag{KE} \label{eq:Kolmogorov}
\end{equation}
or, introducing the tri-diagonal matrix
$$A_N:=\left(
  \begin{array}{cccccc}
    -\alpha_0 & \delta_1 & 0 & \cdots & \cdots & 0 \\
    \beta_0 & -\alpha_1 & \delta_2 & \cdots & \cdots & 0 \\
    0 & \beta_1 & -\alpha_2 & \delta_3 & \cdots & 0\\
    \vdots & \vdots & \ddots & \ddots & \ddots & \vdots\\
    0 & 0 &  \cdots &\beta_{N-2} & -\alpha_{N-1} & \delta_N\\
    0 & 0 & \cdots & 0 & \beta_{N-1} & -\alpha_N
  \end{array}
\right)$$ and the coloumn vector $p(t)=(p_0(t), p_1(t),
\ldots, p_N(t))^T$,
\begin{equation*}
\dot{p}(t)=A_Np(t).
\end{equation*}
The assumption on the tri-diagonality of the matrix can obviously
be weakened, however, most practical problems that motivate our
work  fall into this class. Therefore, to be in line with
applications and to make our results more transparent, the
tri-diagonal case will be considered.

Here we assume that the coefficients $\beta_k$ and $\delta_k$ are \emph{asymptotically
density dependent} in the sense that
\begin{equation}\label{eq:betakdeltak}
    \beta_k= B_N(k), \quad \delta_k= D_N(k),
\end{equation}
and the following limits exist for all $x\in [0,1],$
\begin{equation}\label{eq:BNDN}
    \beta(x)=\lim_{N\to\infty}\frac{B_N(Nx)}{N}, \quad
    \delta(x)=\lim_{N\to\infty}\frac{D_N(Nx)}{N},
\end{equation}
with $\beta$ and $\delta$ (at least) continuous functions, and
\begin{equation}\label{eq:alfakbetakdeltak}
    \alpha_k=\beta_k+\delta_k,\quad k=0,\ldots, N.
\end{equation}
We note that the usually used density dependence means that $ \beta(x)=\frac{B_N(Nx)}{N}$ for all
$N$ (similarly for $D_N$).

In order to get a differential equation approximation of the Markov chain for $N\to \infty$, the
random variables $X_N(t)$, forming a continuous time Markov process, with values in
\begin{equation*}
\l\{0,\frac{1}{N},\frac{2}{N},\dots ,1 \r\}
\end{equation*}
are considered. Then
\begin{equation*}
p_k(t)=P\l(X_N(t)=\frac{k}{N}\r)
\end{equation*}
can be expressed in terms of the transition probabilities as
\begin{equation*}
p_k(t)=\sum_{j=0}^N p_{j,k}(t)p_j(0),
\end{equation*}
where the transition probabilities are
\begin{equation*}
p_{j,k}(t)=P\l(X_N(t)=\frac{k}{N}|\, X_N(0)=\frac{j}{N}\r).
\end{equation*}
Combining this with the Kolmogorov Equation \eqref{eq:Kolmogorov},
it is straightforward to show that the transition matrix is given
by
\begin{equation*}
T_N(t):=\l[p_{j,k}(t)\r] = \e^{A_N^{\top} t}.
\end{equation*}

As a simple calculation shows in Subsection \ref{subsection3c},
the approximating mean-field differential equation to the Markov
chain can be written as
\begin{equation}\label{eq:apprDE}
    \dot x= \beta(x)-\delta(x) .
\end{equation}

The problem of rigorously linking exact stochastic models to
mean-field approximations goes back to the early work of Kurtz
\cite{Ku:69,Ku:70}. Kurtz studied pure-jump density dependent
Markov processes where apart from providing a method for the
derivation of the mean-field model he also used solid mathematical
arguments to prove the stochastic convergence of the exact to the
mean-field model. His earlier results \cite{Ku:69,Ku:70} relied on
Trotter type approximation theorems for operator semigroups. Later
on, the results were embedded in the more general context of
Martingale Theory \cite{EK:05}. A detailed survey of the subject
from the probabilistic point of view is by Darling and Norris
\cite{DANo:08}. Building on this and similar work, McVinish and
Pollett \cite{McVP:11} have recently extended the differential
equation approximation to the case of heterogeneous density
dependent Markov chains, where the coefficients of the transition
rates may vary with the nodes. The results by Kurtz and others in
this area have been cited and extensively used by modelers in
areas such as ecology and epidemiology to justify the validity of
heuristically formulated mean-field models. The existence of
several approximation models, often based on different modelling
intuitions and approaches, has recently highlighted the need to
try and unify these and test their performance against the exact
stochastic models, see House and Keeling \cite{HK:10}. Some steps
in these directions have been made by Ball and Neal, and Lindquist
el al. \cite{BN:08, LMDW:10}, where the authors clearly state the link
between exact and mean-field models.

The probabilistic approach given by Ethier and Kurtz, and Darling and Norris ~in
\cite{EK:05,DANo:08} yields weak or stochastic convergence of the Markov chain to the solution of
the differential equation. Here we have a more moderate goal, namely to prove that the expected
value of the Markov process converges to the solution of the differential equation as
$N\to\infty$, and to prove that the discrepancy between the two is of order $1/N$. The benefit of
framing the question in this simpler or different way lies in a less technical and more accessible
proof. This will mainly rely on well-known results from semigroup theory compared to the
combination of highly specialist tools and results from probability theory. For the applied semigroup methods see also B\'atkai, Csom\'os and Nickel \cite{BCsN:09} and references therein.
This simpler approach, based on the expected value of the Markov chain, is also motivated by
practical considerations, namely that usually the goodness of the approximation is tested by
comparing the average of many individual simulations to the output from the simplified approximate
model. This can be a satisfactory and sufficient comparison since, according to our knowledge,
weak and stochastic converges is rarely tested.

The main result of the paper can be formulated in the following Theorem, where we assume uniformity for the convergence in \eqref{eq:BNDN}.

\begin{thm}
Let
$$
y_1(t)=\sum_{k=0}^N\frac{k}{N} p_k(t)
$$
be the expected value and let $x$ be the solution of (\ref{eq:apprDE}) with initial condition
$x(0)=y_1(0)$. Let us assume that the limits in \eqref{eq:BNDN} are uniform in the sense that
there exists a number $L$, such that for all $x\in [0,1]$
\begin{equation}\label{eq:limBN}
    |\beta(x)-\frac{B_N(Nx)}{N}| \leq \frac{L}{N} .
\end{equation}
Then for any $t_0>0$ there exists a constant $C$, such that
$$\l|x(t)-y_1(t)\r|\leq\frac{C}{N},\; t\in[0,t_0].$$
\end{thm}

Since the essence of the Theorem is well-known we highlight the
\textbf{novelties of our approach in the paper}.

\begin{itemize}
\item The proof is self-contained in the sense that no general, abstract, theorem or combination of theorems are used. The result is
based on the simple fact that if the operators $A_N$ converge to the operator $A$ in a certain
sense as $N\to \infty$, then the operator semigroup $T_N$ generated by $A_N$ converges to the
semigroup $T$ generated by $A$. This result is formulated in Lemma \ref{lem:Kurtz}.

\item The proof automatically implies the rate of convergence, namely it can be shown that the difference between the
expected value and the solution of the mean-field ODE is of order $1/N$.

\item Our tools make it possible to extend the above convergence results from density dependent Markov chains to the more general case of
asymptotically density dependent Markov chains.

\item For Markov chains where the transition rates satisfy some specific sign conditions, a completely new approach is
presented. This is based on deriving a countable system of
ordinary differential equations for the moments of a distribution
of interest and proving a perturbation theorem for this infinite
system.

\end{itemize}

The paper is structured as follows. In Section 2 we motivate our
work via two examples and/or applications: an adaptive network
model with random link activation and deletion, and a $SIS$ type
epidemic model on a static graph. The derivation of ODEs for the
moments is presented in Section 3 together with the heuristic
construction of the mean-filed equation for the first moment. In
Section 4, we present our new approach and use it to prove Theorem
1. In Section 5, we present the derivation of an infinite system
of ODEs for the moments (only in the density dependent case) and
we show how this leads to a new approach that can be used to
derive estimates for all moments, directly from the ODE. This is
contrast with the usual approach where only the expected value of
the Markov chain is estimated.

\section{Motivation}\label{sec:mot}

In this Section we present two important examples that motivate
our investigations.

Recently, it has become more and more important to understand the
relation between the dynamics on a network and the dynamics of the
network, see Gross and Blasius \cite{GB:08}. In the case of
epidemic propagation on networks it is straightforward to assume
that the propagation of the epidemic has an effect on the
structure of the network. For example, susceptible individuals try
to cut their links in order to minimize their exposure to
infection. This leads to a change in network structure which in
turn impacts on how the epidemic spreads. The first step in
modeling this phenomenon is an appropriate dynamic network model
such as the recently proposed globally-constrained Random Link
Activation-Deletion (RLAD) model. This can be described in terms
of Kolmogorov equations as follows,
\begin{align*}
\frac{p_k(t)}{dt}&=\alpha(N-(k-1))\l(1-\frac{k-1}{K_1^{max}}\r) p_{k-1}(t)\\
&-\l[\alpha (N-k)\l(1-\frac{k}{K_1^{max}}\r)+\omega k\r] p_k(t)+\omega (k+1) p_{k+1}(t),\\
k&=0,\ldots, N,
\end{align*}
where $p_k(t)$ denotes the probability that at time $t$ there are
$k$ activated links in the network, and $N$ is the total number of
potential edges. It is assumed that non-active links are activated
independently at random at rate $\alpha$ and that existing links
are broken independently at random at rate $\omega$. Furthermore,
the link creation is globally constrained by introducing a
carrying capacity $K_{1}^{max}$, that is the network can only
support a certain number of edges as given by $K_{1}^{max}$.

Using the above notation, here
\begin{align*}
    \beta_{k}&=\alpha(N-k)\l(1-\frac{k}{K_1^{\max}}\r),\quad \delta_{k}=\omega k, \quad \alpha_k=\beta_{k}+\delta_{k}, \quad k=0,\ldots, N,\\
    \alpha_{-1}&=\delta_{N+1}=0
\end{align*}
with $K_1^{\max}$ being of order $N$. These coefficients clearly
satisfy \eqref{eq:BNDN} and \eqref{eq:alfakbetakdeltak}.
\\

The second motivation comes from epidemiology where a paradigm
disease transmission model is the simple susceptible-infected-susceptible ($SIS$) model on a completely
connected graph with $N$ nodes, i.e.~all individuals are connected
to each other. From the disease dynamic viewpoint, each individual
is either susceptible ($S$) or infected ($I$) -- the susceptible
ones can be infected at a certain rate ($\beta$) if linked to at
least one infected individual and the infected ones can recover at
a given rate ($\gamma$) and become susceptible again. It is known
that in this case the $2^N$-dimensional system of Kolmogorov
equations can be lumped to a $N+1$-dimensional system, see Simon,
Taylor and Kiss \cite{STK:11}.

The lumped Kolmogorov equations take again the form (\ref{eq:Kolmogorov}) with
\begin{align*}
    \beta_{k}&=\beta k(N-k)/N,\quad \delta_{k}=\gamma k, \quad \alpha_k=\beta_{k}+\delta_{k}, \quad k=0,\ldots, N,\\
    \beta_{-1}&=\delta_{N+1}=0.
\end{align*}
These coefficients also satisfy \eqref{eq:BNDN}  and \eqref{eq:alfakbetakdeltak}. We note that in
the case of a homogeneous random graph we get a similar system with a slightly different meaning
of the coefficients.

\section{Momentum approach}\label{sec:momentum}

The basic idea of getting an approximating differential equation
is to calculate the time derivative of the expected value by using
the Kolmogorov equations. Since the obtained equation is not
self-contained, it needs to be closed by using some closure
approximation. In this Section we derive first equations for the
derivatives of every moment, then we briefly show how to get the
simplest mean-field approximation for the first moment. This is
also discussed in the case of asymptotically density dependent
Markov chains.

\subsection{Differential equations for the moments} \label{subsection3a}

Introducing the moments
\begin{equation}
    y_n(t)=\sum_{k=0}^N\l(\frac{k}{N}\r)^np_k(t),\quad n=1,2,\ldots,
    \label{eq:momentum}
\end{equation}
($y_1(t)$ is the expected value we are mainly interested in) we follow Simon and Kiss \cite{KS:10} to derive the
differential equations for $y_n(t)$ starting from the Kolmogorov equations \eqref{eq:Kolmogorov}.
To get the time derivative of $y_n$ the following Lemma will be used.

\begin{lem}
Let $r_k$ ($k=0,1,2,\ldots $) be a sequence and let $r(t)=\sum_{k=0} ^N  r_k p_k(t)$, where $p_k(t)$ is given by \eqref{eq:Kolmogorov}. Then
$$
\dot r = \sum_{k=0} ^N  (\beta_k(r_{k+1}-r_k) + \delta_k(r_{k-1}-r_k)) p_k .
$$
 \label{prop1}
\end{lem}
\begin{proof}
From \eqref{eq:Kolmogorov} we obtain
\begin{align*}
\dot r &= \sum_{k=0}^N r_k \dot p_k = \sum_{k=1}^N r_k \beta_{k-1} p_{k-1}  - \sum_{k=0}^N r_k \alpha_k p_k  + \sum_{k=0}^{N-1} r_k \delta_{k+1} p_{k+1} \\
& = \sum_{k=0}^{N-1} r_{k+1} \beta_{k} p_{k}  - \sum_{k=0}^N r_k \alpha_k p_k  + \sum_{k=1}^{N} r_{k-1} \delta_{k} p_k  .
\end{align*}
Using that $\beta_N=0$, $\delta_0=0$ and $\alpha_k =\beta_k+\delta_k$ we get
$$
\dot r  = \sum_{k=0}^N \left(r_{k+1} \beta_{k} -r_k (\beta_k+\delta_k) + r_{k-1} \delta_{k}\right) p_k  = \sum_{k=0} ^N  (\beta_k(r_{k+1}-r_k) +
\delta_k(r_{k-1}-r_k)) p_k  .
$$
\end{proof}

Before applying Lemma \ref{prop1} with $r_k=(k/N)^n$, it is useful to define the following two new expressions
$$
R_{k,n} = \frac{(k+1)^n -k^n-nk^{n-1}}{N^{n-1}} , \quad Q_{k,n} = \frac{(k-1)^n -k^n+nk^{n-1}}{N^{n-1}} .
$$
Let us introduce
\begin{equation}
d_n(t) =  \sum_{k=0}^N  ( \beta_k R_{k,n} + \delta_k Q_{k,n}) p_k(t) . \label{eq:dn}
\end{equation}
Combining these with Lemma \ref{prop1} leads to
\begin{equation*}
\dot y_n(t) = \sum_{k=0} ^N  \left( \frac{\beta_k}{N} \left( n\frac{k^{n-1}}{N^{n-1}}+
R_{k,n}\right) + \frac{\delta_k}{N} \left( -n\frac{k^{n-1}}{N^{n-1}}+ Q_{k,n}\right) \right)
p_k(t)
\end{equation*}
hence
\begin{equation}
\dot y_n(t)=n\cdot\sum_{k=0}^N\frac{\beta_k-\delta_k}{N}\cdot\l(\frac{k}{N}\r)^{n-1} \cdot
p_k(t)+\frac{1}{N}d_n(t).\label{eq:y_n1}
\end{equation}

Using the binomial theorem $R_{k,n}$ and $Q_{k,n}$ can be expressed in terms of the
powers of $k$, hence $d_n$ can be expressed as $ d_n(t) = \sum_{m=1} ^n d_{nm}y_m(t) $
with some coefficients $d_{nm}$.
The $d_n$ terms contain $N$, hence to use the $1/N \to 0$ limit it has to be shown that $d_n$ remains bounded as $N$ goes to infinity. This is proved in the next lemma.

\begin{lem}
For the functions $d_n$ the following estimates hold
$$
0\leq d_n(t) \leq c\cdot \frac{n(n-1)}{2} \qquad \mbox{ for all }  t\geq 0 .
$$
\label{prop2}
\end{lem}

\begin{proof}
Taylor's theorem, with second degree remainder in Lagrange form, states that
$$
f(x)=f(x_0) + f'(x_0)(x-x_0)+f''(\xi) \frac{(x-x_0)^2}{2} ,
$$
where $\xi$ is between $x_0$ and $x$. This simple result can be used to find estimates for  both $R_{k,n}$ and $Q_{k,n}$.
In particular, applying the above result when $f(x)=x^n$, $x=k+1$ and $x_0=k$ gives
$$
R_{k,n} = \frac{n(n-1)}{2}\frac{\xi^{n-2}}{N^{n-1}}
$$
with some $\xi \in [k,k+1]$. Similarly, when $x=k-1$ and $x_0=k$, we obtain
$$
Q_{k,n} = \frac{n(n-1)}{2}\frac{\eta^{n-2}}{N^{n-1}}
$$
with some $\eta \in [k,k+1]$. Hence, $R_{k,n}$ and $Q_{k,n}$ are non-negative yielding that $d_n(t)\geq 0$. On the other hand, $\xi /N \leq 1$ and $\eta /N \leq 1$ and \eqref{eq:betakdeltak}  lead to the inequality given below
$$
\beta_k R_{k,n} + \delta_k Q_{k,n} \leq \frac{n(n-1)}{2} \left( \frac{\beta_k}{N} + \frac{\delta_k}{N} \right) \leq c\cdot\frac{n(n-1)}{2} .
$$
Hence, the statement follows immediately from \eqref{eq:dn} and using
that $\sum_{k=0}^N p_k(t)=1$.
\end{proof}

We show two possible ways to turn (\ref{eq:y_n1}) into an ODE.
First, we use the approximation $E(F(X))=F(E(X)$ (where $E$ stands
for the expected value and $F$ is a given measurable function) to
derive the mean-field approximation. Then assuming that the Markov
chain is density dependent and the functions $\beta$ and $\delta$
are polynomials, we derive an infinite system of ODEs for the
moments.

\subsection{The mean-field approximation} \label{subsection3c}

Since $d_1=0$, applying (\ref{eq:y_n1}) for $n=1$ we obtain

\begin{equation}\label{eq:dey1}
\dot y_1(t)=\sum_{k=0}^N \l( \frac{B_N(k)}{N}-\frac{D_N(k)}{N} \r) \cdot p_k(t).
\end{equation}

Using the asymptotic density dependence of the Markov chain (\ref{eq:BNDN}), the right-hand side
can be approximated by
$$
\sum_{k=0}^N\l(\beta\l(\frac{k}{N}\r)-\delta\l(\frac{k}{N}\r)\r)\cdot
p_k(t) .
$$

%
%
In order to make the equation ``closed'' the approximation
\begin{equation}
\sum_{k=0}^N\l(\beta\l(\frac{k}{N}\r)-\delta\l(\frac{k}{N}\r)\r)\cdot
p_k(t)\approx \beta\l(\sum_{k=0}^N\frac{k}{N}\cdot
p_k(t)\r)-\delta\l(\sum_{k=0}^N\frac{k}{N}\cdot p_k(t)\r)
\end{equation}
will be used. Substituting this approximation into the equation \eqref{eq:dey1} we obtain the
following differential equation

\begin{equation}
\dot {x} = \beta(x)-\delta(x).\tag{ODE}\label{eq:dey}
\end{equation}
This equation is known as the \emph{mean-field approximation} of the original Kolmogorov equation
\eqref{eq:Kolmogorov}.

\subsection{Infinite system of ODEs in the polynomial case}

In this Subsection it is assumed that the functions in
\eqref{eq:betakdeltak} are polynomials and the Markov chain is
density dependent, that is,
\begin{equation}
    \frac{\beta_k}{N}=\beta\l(\frac{k}{N}\r), \quad \frac{\delta_k}{N}=\delta\l(\frac{k}{N}\r)
    \label{eq:betak}
\end{equation}
and
\begin{equation}
\beta(x)= \sum_{j=0}^l g_j x^j, \qquad  \delta(x)= \sum_{j=0}^l h_j x^j . \label{eq:bdpol}
\end{equation}
Using this and denoting
\begin{equation}
q_j:=g_j-h_j, \quad j=0,1,\ldots ,l,
\end{equation}
from \eqref{eq:y_n1} we obtain that
\begin{align}\label{eq:momentum_eq1}
\dot y_n(t) &=n\cdot\sum_{k=0}^N \sum_{j=0}^l q_j \l(\frac{k}{N}\r)^{n+j-1} \cdot p_k(t) +\frac{1}{N}d_n(t)\notag\\
&= n\cdot \sum_{j=0}^lq_jy_{n+j-1}(t)+\frac{1}{N}d_n(t)
\end{align}

with
\begin{equation}
    0\leq d_n(t)\leq c\cdot \frac{n\cdot (n-1)}{2}.
    \label{eq:d_n}
\end{equation}

Letting $N\to\infty$ on the right-hand-side, we arrive at the
following system
\begin{align}
    \dot{f}_n(t) &=n\cdot\sum_{j=0}^l q_jf_{n+j-1}(t)\notag\\
    n&=1,2,\ldots,\tag{IE}
    \label{eq:fn}
\end{align}
that can be regarded as a system of ``approximating'' differential
equations for \eqref{eq:momentum_eq1}. In Section \ref{sec:InfODE}
we are going to investigate how  $y_1(t)$ can be approximated on
finite time intervals using the solution of this infinite system.

\begin{rem}
All the results for the infinite system obtained from here on
remain valid in the asymptotically density dependent case when
$$
\frac{\beta_k}{N}=\sum_{j=0}^l \tilde{g}_j (N) k^j, \qquad \frac{\delta_k}{N}=\sum_{j=0}^l
\tilde{h}_j(N) k^j
$$
and
\begin{equation*}
    \tilde{g}_j(N)=\frac{g_j}{N^j}+\mathcal{O}\l(\frac{1}{N^{j+1}}\r),\quad \tilde{h}_j(N)=\frac{h_j}{N^j}+\mathcal{O}\l(\frac{1}{N^{j+1}}\r),\quad j=0,\ldots
    ,l.
\end{equation*}
The only difference is that in \eqref{eq:momentum_eq1} we obtain
$$\dot y(t)=n\cdot \sum_{j=0}^lq_jy_{n+j-1}(t)+\frac{1}{N}d_n(t)+\mathcal{O}\l(\frac{1}{N}\r).$$
\end{rem}

\section{Proof of Theorem 1}\label{sec:Kurtz}

In this Section we prove that the solution of \eqref{eq:dey} is an $\mathcal{O}(1/N)$
approximation of the expected value of the Markov chain, that is we prove Theorem 1.

Let us introduce the matrix $A_N$ as in Section \ref{sec:int}
$$A_N:=\left(
  \begin{array}{cccccc}
    -\alpha_0 & \delta_1 & 0 & \cdots & \cdots & 0 \\
    \beta_0 & -\alpha_1 & \delta_2 & \cdots & \cdots & 0 \\
    0 & \beta_1 & -\alpha_2 & \delta_3 & \cdots & 0\\
    \vdots & \vdots & \ddots & \ddots & \ddots & \vdots\\
    0 & 0 &  \cdots &\beta_{N-2} & -\alpha_{N-1} & \delta_N\\
    0 & 0 & \cdots & 0 & \beta_{N-1} & -\alpha_N
  \end{array}
\right).$$

Then the operator families $\l(T_N(t)\r)_{t\geq 0}$ defined as
\begin{equation*}
T_N(t):=\l[p_{j,k}(t)\r] = \e^{A_N^{\top} t}
\end{equation*}
form uniformly continuous semigroups on $\CC^{N+1}$  for each $N\in\NN$.
This yields
\begin{equation}\label{eq:TNhatas}
\l(T_N(t)f\r)\l(\frac{j}{N}\r)=\sum_{k=0}^Nf\l(\frac{k}{N}\r)\cdot p_{j,k}(t)
\end{equation}
for $f=(f_0,\dots ,f_{N})\in\CC^{N+1}$ where we make the identification
$$\CC^{N+1}\equiv\l\{f:f \text{ maps }\{0,\frac{1}{N},\frac{2}{N},\cdots ,1\}\text{ to } \RR\r\}.$$
\vspace{0.5cm}

Assume for the sake of simplicity that on the right-hand-side of \eqref{eq:dey} the functions
$\beta, \delta\in C^2[0,1]$. If we denote the solution of \eqref{eq:dey} with initial condition $x_0$ by
$\varphi(t,x_0)$, then the operator family defined as
\begin{equation}\label{eq:Thatas}
(T(t)f)(x_0):=f(\varphi(t,x_0)),\quad f\in C([0,1]),\; t\geq 0
\end{equation}
defines a strongly continuous operator semigroup on $C([0,1])$ (see Engel, Nagel \cite[Section 3.28]{EN:00}) with generator $(A,D(A))$. We also know that for $f\in C^1([0,1])$

\begin{equation*}
(Af)(x_0)=\l(\beta(x_0)-\delta(x_0)\r)\cdot f'(x_0).
\end{equation*}

The main idea is to approximate the semigroup $\l(T_N(t)\r)_{t\geq 0}$ (that is, the solution of
the transposed \eqref{eq:Kolmogorov}) using the semigroup  $\l(T(t)\r)_{t\geq 0}$ (that is, the
solution of the mean-field equation \eqref{eq:dey}). Observe that the semigroups act on different
spaces: the first one acts on $X_N:=\CC^{N+1}$, the second one on $X:=C^1([0,1])$. In order to
prove an approximation Theorem in a fixed space, we assume that there are linear operators
\begin{align}
&J_N:X_N\to X,\quad J_N(f):=g,\label{eq:J_N}\\
&P_N:X\to X_N,\quad P_N(g):=f\text{ with }f\l(\frac{k}{N}\r)=g\l(\frac{k}{N}\r),\; k=0,\dots ,N\label{eq:P_N}
\end{align}

such that $\|J_N\|\leq 1$, $\|P_N\|\leq 1$, $N\in\NN$, and
\begin{align*}
P_NJ_N&=\Id_{X_N},\quad N\in\NN;\\
J_NP_Nf&\to f,\quad N\to\infty\; \forall f\in X
\end{align*}
are satisfied (see B\'atkai et al. \cite[Definition 3.5]{BCsN:09}). The next Lemma formulates the main approximation result in rigorous terms.

\begin{lem}\label{lem:Kurtz}
Assume that the conditions of Theorem 1 are satisfied. For $\l(T_N(t)\r)_{t\geq 0}$ and $\l(T(t)\r)_{t\geq
0}$ the following holds: for all $f\in C^2([0,1])$ and $t_0>0$ there exists $C=C(f,t_0)>0$ such
that for all $t\in [0,t_0]$
\begin{equation}
\|(P_NT(t)-T_N(t)P_N)f\|\leq\frac{C}{N},
\end{equation}
where $P_N$ denotes the projection defined in \eqref{eq:P_N}.
\end{lem}
\begin{proof}
For the generators $A_N$, $A$, and for $f\in C^2([0,1])$ the
following identities hold:
\begin{align*}
(P_NAf)\l(\frac{k}{N}\r)&=\l(\beta\l(\frac{k}{N}\r)-\delta\l(\frac{k}{N}\r)\r)\cdot f'\l(\frac{k}{N}\r);\\
(A_NP_Nf)\l(\frac{k}{N}\r)&=\frac{\beta_k}{N}\cdot\frac{f(\frac{k+1}{N})-f(\frac{k}{N})}{\frac{1}{N}}-\frac{\delta_k}{N}\cdot\frac{f(\frac{k}{N})-f(\frac{k-1}{N})}{\frac{1}{N}},
\end{align*}
$k=0,\dots ,N.$ The idea of proving the estimate for the semigroups is to estimate the difference
of the generators. Their difference can be divided into two parts as follows.
\begin{align*}
(P_NAf)\l(\frac{k}{N}\r) - (A_NP_Nf)\l(\frac{k}{N}\r) &= \l(
\beta\l(\frac{k}{N}\r)-\frac{\beta_k}{N}
-\delta\l(\frac{k}{N}\r)+\frac{\delta_k}{N}\r) f'\l(\frac{k}{N}\r) \\
+ \frac{\beta_k}{N} \l(f'(\frac{k}{N}) -
\frac{f(\frac{k+1}{N})-f(\frac{k}{N})}{\frac{1}{N}} \r) &+
\frac{\delta_k}{N} \l(
\frac{f(\frac{k}{N})-f(\frac{k-1}{N})}{\frac{1}{N}} -
f'(\frac{k}{N}) \r) .
\end{align*}
The first part can be estimated by using the asymptotic density dependence \eqref{eq:limBN} as
\begin{equation}\label{eq:betakNestimate}
\l| \beta\l(\frac{k}{N}\r)-\frac{\beta_k}{N} \r| =
\l|\beta\l(\frac{k}{N}\r)-\frac{B_N(N\frac{k}{N})}{N}\r| \leq
\frac{L}{N}
\end{equation}
and
\begin{equation}\label{eq:deltakNestimate}
\l| \delta\l(\frac{k}{N}\r)-\frac{\delta_k}{N} \r| =
\l|\delta\l(\frac{k}{N}\r)-\frac{D_N(N\frac{k}{N})}{N}\r| \leq
\frac{L}{N}
\end{equation}
and using that $f'$ is bounded.

The second part can be estimated by using Taylor's formula for $f\in C^2([0,1])$. We obtain that for each $k=0,\dots ,N$ there exists $\xi_k\in (\frac{k}{N},\frac{k+1}{N})$ such that
\begin{equation}
\l|\frac{f(\frac{k+1}{N})-f(\frac{k}{N})}{\frac{1}{N}}-f'(\frac{k}{N})\r|=\l|N\cdot
\frac{f''(\xi_k)}{2}\cdot \frac{1}{N^2}\r|\leq
\frac{1}{2N}\l\|f''\r\|.
\end{equation}
Since $\beta$ and $\delta$ are bounded on $[0,1]$ (they are in $C^2[0,1]$), we obtain from \eqref{eq:betakNestimate} and \eqref{eq:deltakNestimate} that $\frac{\beta_k}{N}$, $\frac{\delta_k}{N}$ are uniformly bounded for all $k=0,\dots ,N$ and $N\in\NN$. Hence, for all $f\in C^2([0,1])$ there exists $K>0$ such that
\begin{equation}\label{eq:generatordiff}
\l\|\l(P_NA-A_NP_N\r)f\r\|\leq \frac{K}{N}\|f''\|+\frac{L}{N}\|f'\|.
\end{equation}

It is easy to see  that for all $N$, the operators $\widetilde{T}_N(t):=J_N T_N(t)P_N$, $t\geq 0$ form a strongly continuous semigroup on $X$ -- where $J_N$ is defined in \eqref{eq:J_N} --  with generator $\widetilde{A}_N:=J_N A_NP_N$.

Using the variation of parameters formula (see e.g.~Engel, Nagel
\cite[Corollary III.1.7]{EN:00}), for the difference of the
(projected) semigroups, and for $f\in C^2([0,1])$ we obtain that
\begin{align*}
\l(P_NT(t)-T_N(t)P_N\r)f&=P_N\l(T(t)f-\widetilde{T}_N(t)f\r)\\
&=\int_0^tP_N \widetilde{T}_N(t-s)\l(A-\widetilde{A}_N\r)T(s)f\ds\\
&= \int_0^tT_N(t-s)\l(P_NA-A_NP_N\r)T(s)f\ds.
\end{align*}
By the estimate in \eqref{eq:generatordiff} and using that $T(t)$ maps the subspace $C^2([0,1])$
of $X$ into itself (see Chicone \cite[Theorem 1.3]{Chi:06}), we obtain that for $f\in C^2([0,1])$ there
exist $K^*>0$ and $L^*>0$, such that
$$\l\|\l(P_NT(t)-T_N(t)P_N\r)f\r\|\leq \int_0^t \frac{K^*}{N}\cdot \|(T(s)f)''\| + \frac{L^*}{N}\cdot \|(T(s)f)'\| \ds.$$
Since $\beta, \delta\in C^2([0,1])$, the solution function $\varphi$ of \eqref{eq:dey} is also a $C^2$-function of the initial data.  Hence, for $s\in[0,t]$ there exists $K_t>0$ such that
$$
\l\|(T(s)f)''\r\|=\l\|(f\circ\varphi(s,\cdot))''\r\|=\l\|(f''\circ \varphi)\cdot
(\varphi')^2+(f'\circ \varphi)\cdot \varphi''\r\|\leq K_t\cdot \|f\|_{C^2}
$$
and similarly $\l\|(T(s)f)'\r\| \leq K^*_t\cdot \|f\|_{C^2}$. This yields that for $f\in
C^2([0,1])$ and $t_0>0$ there exists $C=C(f,t_0)>0$ such that for all $t\in [0,t_0]$
\begin{equation*}
\|\l(P_NT(t)-T_N(t)P_N\r)f\|\leq\frac{C}{N}.
\end{equation*}
\end{proof}

Using this semigroup approximation Lemma we can now prove Theorem 1.

\begin{proof}[Proof of Theorem 1]
We apply Lemma \ref{lem:Kurtz} for $f=\mathrm{id}_{[0,1]}.$ We obtain that for any $t_0> 0$ there
exists $C=C(\mathrm{id},t_0)$ such that for all $t\in [0,t_0]$
\begin{equation}\label{eq:thmSIS_biz}
\l\|\l(P_NT(t)-T_N(t)P_N\r)\mathrm{id}\r\|\leq\frac{C}{N}.
\end{equation}
Observe that by \eqref{eq:momentum} and \eqref{eq:TNhatas}, the
following relation holds
\begin{align*}
y_1(t)&=\sum_{k=0}^N\frac{k}{N}\cdot p_k(t)=\sum_{k=0}^N\frac{k}{N}\cdot \sum_{j=0}^Np_{j,k}(t)\cdot p_j(0)\\
&=\sum_{j=0}^Np_j(0)\sum_{k=0}^N\mathrm{id}(\frac{k}{N})\cdot p_{j,k}(t)=\sum_{j=0}^Np_j(0)\cdot\l(T_N(t)P_N\mathrm{id}\r)(\frac{j}{N})\\
&=\langle p(0),T_N(t)P_N\mathrm{id}\rangle.
\end{align*}
Furthermore, using \eqref{eq:Thatas} it is easy to show that
$$\l(P_NT(t)\mathrm{id}\r)(\frac{k}{N})=\varphi(t,\frac{k}{N}).$$
It can be seen that it is enough to prove the statement when the
initial condition is $p_m(0)=1$, $p_j(0)=0$, $j\neq m$. Then
$y_1(0) = \frac{m}{N}$ yielding $x(0) = \frac{m}{N}$ and therefore
$x(t)=\varphi\l(t,\frac{m}{N}\r)$. Hence, combining the above
facts we obtain
\begin{align*}
\l|y_1(t) - x(t)\r|&=\l|\langle p(0),T_N(t)P_N\mathrm{id}\rangle-\varphi(t,\frac{m}{N})\r|\\
&=\l|\l(T_N(t)P_N\mathrm{id}\r)(\frac{m}{N})-\l(P_NT(t)\mathrm{id}\r)(\frac{m}{N})\r|\leq\frac{C}{N}
\end{align*}
where we used \eqref{eq:thmSIS_biz}.
\end{proof}

\section{Infinite system of ODEs for the moments} \label{sec:InfODE}

In this Section we consider the infinite system of ODEs
\eqref{eq:momentum_eq1} for the moments and its formal limit as
$N\to \infty$ \eqref{eq:fn}. Our aim here is to prove that the
solutions of the first system converge to those of the second as
$N\to \infty$ and $t$ is in a bounded time interval. That is, that the
moments of the Markov chain converge to the solutions of the
infinite system \eqref{eq:fn}. As a bi-product of this result we
get a completely new proof of Theorem 1 under extra sign
conditions on the transition rates.


In order to prove this general convergence result were are going to use the approach of Kato
\cite{Kat:54} and Banasiak et al.~\cite{BLN:06}.

\begin{lem}\label{thm:generator}
Denote the formal operator
\begin{equation*}
\l(\sL f\r)_n:=n\cdot\sum_{j=0}^l q_jf_{n+j-1},\quad n=1,2,\ldots.
\end{equation*}
Let
\begin{equation*}
\sD(L_{\max}):=\l\{f\in\ell^1:\sL f\in\ell^1\r\}, \quad L_{\max}:=\sL |\sD(L_{\max}).
\end{equation*}
Assume that $q_0,q_2,\ldots,q_l\geq 0$, $q_1\leq 0$ and
\begin{align}
q_0+q_1+\cdots +q_l&\leq 0;\\
q_0-2q_3-3q_4-4q_5-\cdots -(l-1)q_l&\leq 0.
\end{align}
Then there exists an operator $(L,D)$ such that $L\subset L_{\max}$ and $(L,D)$ is the generator
of a positive strongly continuous semigroup $\l(T(t)\r)_{t\geq 0}$ of contractions on $\ell^1.$
\end{lem}
\begin{proof}
Building the infinite coefficient matrix for $\sL$, the assumptions imply that all the column sums
are less or equal than $0$:
\begin{equation}\label{column_sums}
(n+1)q_0+nq_1+(n-1)q_2+(n-2)q_3+\cdots +(n-l+1)q_l\leq 0,\quad n=1,2,\ldots.
\end{equation}
Hence, using the assumptions on the signs of the $q_j$'s, we can apply Theorem 1 of Kato in \cite{Kat:54}.
Though the assumptions of Kato's theorem are referred to the equality instead of the inequality in
\eqref{column_sums}, the proof works without any changes in our case. Thus, we do not repeat the
proof here.
\end{proof}

Let us turn to our original question on the expected value $y_1(t)$ -- that is, the first
coordinate of the solution of \eqref{eq:momentum_eq1}.
\begin{thm}
Assume that the conditions of Lemma \ref{thm:generator} are satisfied. Then there exists an
appropriate Banach space of sequences $\ell^1_w$ with norm $\|\cdot\|_w$ such that for any initial
condition $y_0\in\ell^1_w$ the solution $y(t):=\l(y_n(t)\r)_{n\in\NN}$ of \eqref{eq:momentum_eq1}
and the solution $f(t):=\l(f_n(t)\r)_{n\in\NN}$ of \eqref{eq:fn} satisfy
\begin{equation}\label{eq:yt-ft}
\l\|y(t)-f(t)\r\|_w\leq\frac{K}{N}\quad \text{ for }t\in[0,T],
\end{equation}
that is the solution of \eqref{eq:momentum_eq1} tends to the solution of \eqref{eq:fn} in a finite
time interval as $N\to\infty.$
\end{thm}
\begin{proof}
Using the formal operator $\sL$ as in Lemma \ref{thm:generator}, the system
\eqref{eq:momentum_eq1} can be written as
\begin{align}\label{eq:yt}
\dot{y}(t)&=\sL y(t)+\frac{1}{N}d(t),\notag\\
y(0)&=y_0,
\end{align}
where $d(t):=\l(d_n(t)\r)_{n\in\NN}$.  Denote by $\ell^1_w$ a weighted $\ell^1$ space of sequences
such that $d(t)\in\ell^1_w$, $t\in\RR_+$, e.g.,
$$\l\|x\r\|_w:=\sum_{n=1}^{\infty} |x_n|r^n\text{ for }x=(x_n)_{n\in\NN}\text{ with }|r|<1.$$In this case, by \eqref{eq:d_n},
$$\l\|d(t)\r\|_w\leq c\cdot \sum_{n=1}^{\infty} n(n-1)r^n=:K<\infty,\quad t\in\RR_+.$$
Applying Lemma \ref{thm:generator} on the space $\ell^1_w$, we obtain an operator $(L_w,D_w)$ --
acting formally as $\sL$ -- which is the generator of a contraction semigroup $\l(T_w(t)\r)_{t\geq
0}$ on $\ell^1_w$. By the formula of variation of parameters for \eqref{eq:yt} (see, e.g., Engel and Nagel
\cite[Corollary III.1.7]{EN:00}), we have
\begin{align}
y(t)&=T_w(t)y_0+\frac{1}{N}\cdot\int_0^tT_w(t-s)d(s)\, ds\\
&=f(t)+\frac{1}{N}\cdot\int_0^tT_w(t-s)d(s)\, ds
\end{align}
for any $y_0\in\ell^1_w.$ Thus
\begin{equation}
\l\|y(t)-f(t)\r\|_w\leq\frac{1}{N}\cdot \int_0^t\l\|T_w(t-s)\r\|_w \cdot \l\|d(t)\r\|_w\, ds\leq
\frac{1}{N}\cdot t\cdot K,
\end{equation}and this implies \eqref{eq:yt-ft}.
\end{proof}

\begin{rem}
We note that the sign conditions of Lemma \ref{thm:generator} hold in the case of globally
constrained random link activation-deletion process, but they do not hold in the case of an SIS
epidemic on a complete graph.
\end{rem}

\section*{Acknowledgments}
The project was supported by the European Union and co-financed by the European Social Fund (grant agreement no. TAMOP
4.2.1./B-09/1/KMR-2010-0003). Supported by the OTKA grant Nr. K81403. A.~B\'atkai was further supported by the Alexander von Humboldt-Stiftung. E.~Sikolya was supported by the Bolyai Grant of the Hungarian Academy of Sciences.

\end{document}